\documentclass{amsart}
\usepackage[utf8]{inputenc}
\usepackage{bbold}
\usepackage{amssymb}
\usepackage{amsmath}
\usepackage{amsthm}
\usepackage{mathtools}
\theoremstyle{definition}
\newtheorem{definition}{Definition}
\newtheorem{theorem}{Theorem}
\newtheorem{lemma}{Lemma}
\newtheorem{example}{Example}
\newtheorem{proposition}{Proposition}
\newtheorem{notation}{Notation}

\newtheorem{remark}{Remark}

\title{Generalised Fibonacci sequences constructed from balanced words}

\author{Kevin G. Hare}

\address{Department of Pure Mathematics, University of Waterloo, 200 University Avenue West\\
Waterloo, Ontario N2L 3G1, Canada}
\email{kghare@uwaterloo.ca}
\thanks{NSERC grant 2019-03930}

\author{J.C. Saunders}

\address{Department of Mathematics and Statistics, Mathematical Sciences 476
University of Calgary, 2500 University Drive NW\\
Calgary, Alberta T2N 1N4, Canada}
\email{john.saunders1@ucalgary.ca}
\thanks{Azrieli International Postdoctoral Fellowship, Postdoctoral Fellowship at University of Calgary}

\begin{document}

\maketitle

\begin{abstract}
We study growth rates of generalised Fibonacci sequences of a particular structure. These sequences are constructed from choosing two real numbers for the first two terms and always having the next term be either the sum or the difference of the two preceding terms where the pluses and minuses follow a certain pattern. In 2012, McLellan proved that if the pluses and minuses follow a periodic pattern and $G_n$ is the $n$th term of the resulting generalised Fibonacci sequence, then
\begin{equation*}
\lim_{n\rightarrow\infty}|G_n|^{1/n}
\end{equation*}
exists. We extend her results to recurrences of the form $G_{m+2} = \alpha_m G_{m+1} \pm G_{m}$ if the choices of pluses and minuses, and of the $\alpha_m$ follow a balancing word type pattern.
\end{abstract}


\section{Introduction}
The Fibonacci sequence, recursively defined by $f_1=f_2=1$ and $f_n=f_{n-1}+f_{n-2}$ for all $n\geq 3$, has been generalised in several ways. In 2000, Divakar Viswanath studied random Fibonacci sequences given by $t_1=t_2=1$ and $t_n=\pm t_{n-1}\pm t_{n-2}$ for all $n\geq 3$.
Here each $\pm$ is chosen to be $+$ or $-$ with probability $1/2$, and 
    are chosen independently.
Viswanath proved that
\begin{equation*}
\lim_{n\rightarrow\infty}\sqrt[n]{|t_n|}=1.13198824\dots
\end{equation*}
with probability $1$ where the exact value of the above limit is given as the exponential function on an integral expression involving a measure defined on Stern–Brocot intervals \cite{viswanath}. Almost nothing is known about this constant, however, not even if it is irrational. One of the key ideas in his proof was to study random Fibonacci sequences by using products of matrices. More specifically, let
\begin{equation*}
A:=\begin{bmatrix}
0 & 1\\
1 & 1
\end{bmatrix}
\text{ and }
B:=\begin{bmatrix}
0 & 1\\
1 & -1
\end{bmatrix}.
\end{equation*}
Then for all $n\in\mathbb{N}$ the $(n+1)$th and $(n+2)$th terms of a random Fibonacci sequence satisfy
\begin{equation*}
[1,1]Q_n=[G_{n+1},G_{n+2}]
\end{equation*}
where $Q_n$ is a matrix product consisting of $n$ $A$s and $B$s as factors where the pattern of $A$s and $B$s reflect the pattern of pluses and minuses generating the random Fibonacci sequence in question. 

In 2006, Jeffrey McGowan and Eran Makover used the formalism of trees to evaluate the growth of the average value of the $n$th term of a random Fibonacci sequence \cite{eran}, which by Jensen's inequality is larger than Viswanath's constant. More precisely, they proved that
\begin{equation*}
1.12095\leq\sqrt[n]{E(|t_n|)}\leq 1.23375
\end{equation*}
where $E(|t_n|)$ is the expected value of the $n$th term of the sequence. In 2007, Rittaud \cite{rittaud} improved this result and obtained
\begin{equation*}
\lim_{n\rightarrow\infty}\sqrt[n]{E(|t_n|)}=\alpha=1\approx 1.20556943\ldots
\end{equation*}
where $\alpha$ is the only real root of $f(x)=x^3-2x^2-1$. In 2010, Janvresse, Rittaud, and De La Rue generalised Viswanath's result to random Fibonacci sequences involving different coefficients and probabilities of the choices of pluses and minuses \cite{janvresse}. In 2012, Karyn McLellan used Viswanath's idea of representing random Fibonacci sequences as matrix products with the matrices $A$ and $B$ as factors to study sequences that begin with two real numbers with the next term always being  either the sum or the difference of the two preceding terms, but where the pattern of the pluses and minuses was periodic instead of random \cite{mclellan}. McLellan determined the growth rate of any such sequence, showing that Viswanath's limit still exists, albeit evaluating to different values, depending on the particular sequence in question. She used these growth rates to provide an alternative method of calculation for Viswanath's constant in the random case.

In 2018, the authors, in \cite{hare}, extended Rittaud's results and determined the probability that a random infinite walk down the tree contains no $(1,1)$ pairs after the initial root. We also determined tight upper and lower bounds on the number of coprime $(a,b)$ pairs at any given depth in the tree for any coprime pair $(a,b)$.

In this paper we consider a more general model.  
Starting with $G_1$ and $G_2$ as any real numbers, consider the recurrence
    $G_{m+2} = \alpha_m G_{m+1} \pm G_{m}$ where the $\alpha_m$ are taken from a finite
    set.  
This can be modeled by matrix multiplication as
\[ [G_{m+1},G_{m+2}] = [G_{m} , G_{m+1}] \begin{bmatrix} 0 & \pm 1 \\ 1 & \alpha_m 
    \end{bmatrix}. \]
Here we extend McLellan's results and show that Viswanath's limit exits if the 
     pattern of matrix multiplications generating the sequence $\left(G_n\right)_n$ follows certain balancing word patterns. Recall that
\begin{definition}
A \textit{balanced word} or \textit{Sturmian word} $w$ is an infinite word over a two letter alphabet $\{a,b\}$ such that, for any two subwords from $w$ of the same length, the number of letters that are $a$ in each of these two subwords will differ by at most $1$. 
\end{definition}
We consider the following construction of a balanced or Sturmian word throughout this paper. There exists a sequence of nonnegative integers $q_0,q_1,a_2,\ldots$ with $q_i>0$ for all $i>0$ such that the infinite set of words $\{s_n\}_{n\geq 0}$ constructed from $s_0=b$, $s_1=a$, and $s_{n+1}=s_n^{q_{n-1}}s_{n-1}$ for all $n\in\mathbb{N}$ converges to $w$. Strumian or balanced words have been great studied, and two relevant references are Allouche and Shallit \cite{allouche} and de Luca \cite{deLuca}. Instead of having the pluses and minuses follow a periodic pattern studied by McLellan, we will have the pluses and minuses follow the pattern in an infinite balanced word.
\begin{notation}\label{P_1P_2}
Let $v\geq 1$, and let 
\[
A_i := \begin{bmatrix} 0 & \epsilon_i \\ 1 & \alpha_i \end{bmatrix} 
\]
for $i = 1, 2, \dots, v$, 
where $\epsilon_i \in \{-1, 1\}$ and $\alpha_i \in \mathbb{Z}$.
We note that each $A_i$ has determinant $\pm 1$.
Let $P_1$ and $P_2$ be products of these matrices of length $k_1$ and $k_2$ respectively, allowing multiplicity.
That is 
\begin{equation*}
P_j:=\begin{bmatrix}
a_j & b_j\\
c_j & d_j
\end{bmatrix}
= A_{j,1} A_{j,2} \dots, A_{j,k_j}.
\end{equation*}
We further require for $j = 1, 2$ that 
\begin{equation}\label{inequalitiesentries}
b_j,c_j\neq 0,\left|d_j\right|\geq 2, \left|a_j\right|\leq\left|b_j\right|\leq\left|d_j\right|\text{ and }\left|a_j\right|\leq\left|c_j\right|\leq\left|d_j\right|.
\end{equation}
For a sequence of positive integers $(q_m)_{m\in\mathbb{N}}$ define $P_m$, $k_m$ and $A_{m, j_1}, \dots A_{m, j_{k_m}}$,  
    inductively as
\begin{align*}
P_{m+2} & := P_{m+1}^{q_m} P_m \\
        & := \underbrace{\underbrace{A_{m+1, j_1} \dots A_{m+1, j_{k_{m+1}}}}_{P_{m+1}} \dots
             \underbrace{A_{m+1, j_1} \dots A_{m+1, j_{k_{m+1}}}}_{P_{m+1}}}_{q_m \text{times}}
             \underbrace{A_{m, j_1} \dots A_{m+1, j_{k_{m}}}}_{P_{m}} \\
        & := A_{m+2, j_1} \dots A_{m+2, j_{k_{m+2}}} \\
        & := \begin{bmatrix}
               a_{m+2}& b_{m+2}\\
               c_{m+2} & d_{m+2}
             \end{bmatrix}.
\end{align*}
Notice that the sequence $\left(P_m\right)_m$ is a standard Sturmian word (or balanced word) on the alphabet $\{P_1,P_2\}$. We observe that for all $m\geq 3$ and $\ell\geq 1$ that $P_{m+\ell}$ always starts by $P_m$, which is a product of $k_m$ matrices. Thus, whenever $n\leq k_m$, the first $n$ matrices of $P_{m+\ell}$ is independent of the choice of $\ell\in\mathbb{N}$. With this observation, we define 
\begin{equation*}
Q_n:= A_{m, 1} \dots A_{m, n} 
   := \begin{bmatrix}
      e_n & f_n\\
      g_n & h_n
\end{bmatrix}
\end{equation*}
for $k_m \geq n$. Notice that so long as $k_m\geq n$, $Q_n$ is independent of the choice of $m$. Finally, let $G_1$ and $G_2$ be any two real numbers and for every $n\in\mathbb{N}$ let $G_{n+2}=\alpha_{m,n}G_{n+1}+\epsilon_{m,n}G_n$ where
\begin{equation*}
A_{m,n}
   := \begin{bmatrix}
      0 & \epsilon_{m,n}\\
      1 & \alpha_{m,n}
\end{bmatrix}
\end{equation*}
Notice that for all $n\in\mathbb{N}$ we have
\begin{equation*}
[G_1,G_2]Q_n=[G_{n+1},G_{n+2}].
\end{equation*}
\end{notation}
We prove the following:
\begin{theorem}\label{bigthm}
Let $q_m$, $P_m$, $a_m$, $c_m$, $Q_n$, and $G_n$ be as defined in Notation \ref{P_1P_2}. Then $\lim_{m\rightarrow\infty}\frac{a_m}{c_m}=\lim_{m\rightarrow\infty}\frac{b_m}{d_m}$ exists, and is positive, and is either $1$ or irrational. Let this limit be denoted by $M$. 
\begin{enumerate}
\item If $M=1$, then $|G_n|$ grows at most linearly, i.e. there exists $C>0$ such that
\begin{equation*}
|G_n|<Cn
\end{equation*}
for all $n\in\mathbb{N}$.
\item If $M$ is irrational and $G_1\neq\frac{-G_2}{M}$, then
\begin{equation}\label{Gnlimit}
\lim_{n\rightarrow\infty}|G_n|^{1/n}
\end{equation}
exists with this limit being greater than $1$.
\end{enumerate}
\end{theorem}
The proof of Theorem \ref{bigthm} is divided up in the subsequent sections of the paper as follows. In Section \ref{sec2} we prove some necessary technical lemmas on the entries in the matrices $\left(P_m\right)_{m\in\mathbb{N}}$. We then divide the proof up into two cases. If $\left|b_i\right|=\left|c_i\right|=\left|d_i\right|-1=\left|a_i\right|+1$ for all sufficiently large $i$, then it turns out that $M=1$ and we get the case that $\left|G_n\right|$ grows at most linearly. This is proved in Section \ref{sec3}. Otherwise we get $M$ is irrational and the exponential growth of $\left|G_n\right|$, which is dealt in Section \ref{sec4}. For the rest of the present section though we include some remarks on Theorem \ref{bigthm} and illustrate them with an example.
\begin{remark}
If we restrict $G_1,G_2\in\mathbb{Z}$ or even to just $G_1,G_2\in\mathbb{Q}$, then we can ignore the condition that $G_1\neq\frac{-G_2}{M}$ in Theorem \ref{bigthm} since $M$ is irrational.
\end{remark}
\begin{remark}\label{alphal}
Let $0<\alpha<1$ be irrational. If in Notation \ref{P_1P_2} we pick the sequence of positive integers $q_1,q_2,q_3,\ldots $ so that the continued fraction expansion of $\alpha$ can be represented as $[0;q_1,q_2,q_3,\ldots]$, then we have
\begin{equation}\label{alphalimit}
\lim_{m\rightarrow\infty}\frac{\text{number of }P_2\text{s in }P_m}{\text{number of }P_1\text{s and }P_2\text{s in }P_m}=\alpha.
\end{equation}
See, for example, \cite{deLuca}. 
\end{remark}
We give some examples of $A_1$, $A_2$, $P_1$, and $P_2$ that satisfy Notation \ref{P_1P_2} with $v=2$.
\begin{remark}
For all $a,b,c,d\in\mathbb{R}$ observe that
\begin{equation*}
\begin{bmatrix}
0 & 1\\
-1 & 0
\end{bmatrix}
\cdot
\begin{bmatrix}
a & b\\
c & d
\end{bmatrix}
=
\begin{bmatrix}
0 & -1\\
1 & 0
\end{bmatrix}
\cdot
\begin{bmatrix}
d & -b\\
-c & a
\end{bmatrix}.
\end{equation*}
It follows that if we replace the inequalities in \eqref{inequalitiesentries} with the inequalities
\begin{equation*}
0\leq\frac{|d_i|}{|b_i|},\frac{|c_i|}{|a_i|},\frac{|d_i|}{|c_i|},\frac{|b_i|}{|a_i|}\leq 1,
\end{equation*}
then Theorem \ref{bigthm} still holds.
\end{remark}
\begin{example}\label{entriessize}
Let 
\begin{equation*}
A:=\begin{bmatrix}
0 & 1\\
1 & 1
\end{bmatrix}
\text{ and }
B:=\begin{bmatrix}
0 & 1\\
1 & -1
\end{bmatrix}.
\end{equation*}
Let $P_1$ and $P_2$ be a product matrices of matrices of the form $A^j$ and $B^k$ where $j,k\geq 2$. Then the matrices $A$, $B$, $P_1$, and $P_2$ satisfy the matrices $A_1$, $A_2$, $P_1$, and $P_2$ respectively in Notation \ref{P_1P_2} with $v=2$ with $P_1$ and $P_2$ satisfying \eqref{inequalitiesentries}.
\end{example}
\begin{remark}
In Theorem \ref{bigthm} it is possible that if the sequence $\left(G_n\right)_n$ grows at most linearly, then it could contain a bounded infinite subsequence of terms. For example, let the matrices $A$ and $B$ be as in Example \ref{entriessize} and let $P_1=P_2=A^3B^3$. Then for all $k\in\mathbb{N}$ we can verify that
\begin{equation*}
Q_{6k+1}=(A^3B^3)^kA=(-1)^k\begin{bmatrix}
4k & 1\\
4k+1 & 1
\end{bmatrix}.
\end{equation*}
We can thus deduce that $|G_{6k+3}|=|G_1+G_2|$ for all $k\in\mathbb{N}$. It is routine to check that the entries in $Q_n$ in Example \ref{entriessize} grow at most linearly and so any corresponding Fibonacci sequence will grow at most linearly.
\end{remark}
We prove that the matrices in Example \ref{entriessize} satisfy Notation \ref{P_1P_2} and Theorem \ref{bigthm} in the next section. We give a concrete example though of matrices $P_1$ and $P_2$ and a sequence of positive integers $(q_m)_{m\in\mathbb{N}}$ for illustration.
\begin{example}
Let $P_1=A^2$ and $P_2=B^2$. Consider the number $1/\pi$, which has continued fraction expansion $[0;3,7,15,1,\ldots]$. Let our sequence of positive integers  $(q_m)_{m\in\mathbb{N}}$ be these convergents so that $q_1=3$, $q_2=7$, $q_3=15$, $q_4=1,\ldots$ Then $P_1$ and $P_2$ satisfy Notation \ref{P_1P_2} and \eqref{inequalitiesentries}. Then we have $P_3=B^6A^2$, $P_4=(B^6A^2)^7B^2,\ldots$ Let $G_1=G_2=1$ in Theorem \ref{bigthm}. Then the corresponding Fibonacci sequence starts out as follows.
\begin{align*}
G_3&=G_1-G_2=1-1=0\\
G_4&=G_2-G_3=1-0=1\\
G_5&=G_3-G_4=0-1=-1\\
G_6&=G_4-G_5=1-(-1)=2\\
G_7&=G_5-G_6=-1-2=-3\\
G_8&=G_6-G_7=2-(-3)=5\\
G_9&=G_7+G_8=-3+5=2\\
G_{10}&=G_8+G_9=5+2=7\\
G_{11}&=G_9-G_{10}=2-7=-5\\
G_{12}&=G_{10}-G_{11}=7-(-5)=12\\
\ldots&
\end{align*}
Also, we have
\begin{equation*}
P_3=\begin{bmatrix}
-3 & -11\\
5 & 18
\end{bmatrix}
\text{ and }
P_4=\begin{bmatrix}
88364872 & -21089221\\
-144059117 & 343812479
\end{bmatrix}
\end{equation*}
and notice that
\begin{equation*}
\frac{1}{2}<\frac{3}{5},\frac{11}{18},\frac{88364872}{144059117},\frac{21089221}{343812479}<\frac{2}{3}.
\end{equation*}
Then by induction $n\in\mathbb{N}$ using Lemma \ref{positive} we have that if
\begin{equation*}
P_n=\begin{bmatrix}
a_n & b_n\\
c_n & d_n
\end{bmatrix},
\end{equation*}
then
\begin{equation*}
\frac{1}{2}<\frac{\left|a_n\right|}{\left|c_n\right|},\frac{\left|b_n\right|}{\left|d_n\right|}<\frac{2}{3}
\end{equation*}
for all $n\geq 3$. Using Lemma \ref{limitsinfinity}, we therefore have that $|b_m|=|c_m|=|d_m|-1=|a_m|+1$ cannot hold for sufficiently large $m$ so that from the work in Section \ref{sec4} the sequence $\left(G_n\right)_n$ grows exponentially with the limit in \eqref{Gnlimit} existing and greater than $1$. As well, we can deduce from Remark \ref{alphal} that the fraction of $+$'s creating the Fibonacci sequence tends to $1/\pi$.
\end{example}
\section{Preliminary Results}\label{sec2}
To prove Theorem \ref{bigthm} we first require some preliminary lemmas.
\begin{lemma}\label{positive}
Suppose we have
\begin{equation*}
\begin{bmatrix}
a_1 & b_1\\
c_1 & d_1
\end{bmatrix}
\cdot
\begin{bmatrix}
a_2 & b_2\\
c_2 & d_2
\end{bmatrix}
=
\begin{bmatrix}
a_3 & b_3\\
c_3 & d_3
\end{bmatrix}
\end{equation*}
where $a_i,b_i,c_i,d_i\in\mathbb{Z}$ for $i=1,2,3$ with $c_1,d_1,b_2,d_2$ nonzero and that the determinants of all the matrices are $1$. Suppose that $|d_1|\geq|c_1|$, $|b_1|\geq|a_1|$, $|d_2|\geq|b_2|$, $|c_2|\geq|a_2|$ and that 
\begin{equation*}
\frac{r_1}{r_2}\leq\frac{|a_1|}{|c_1|},\frac{|b_1|}{|d_1|}\leq\frac{r_3}{r_4}
\end{equation*}
where $r_i\in\mathbb{Z}$ for $i=1,2,3,4$ with $|d_1|>r_2$ and $|d_1|>r_4$. Then
\begin{equation*}
\frac{r_1}{r_2}\leq\frac{|a_3|}{|c_3|},\frac{|b_3|}{|d_3|}\leq\frac{r_3}{r_4}.
\end{equation*}
Also, suppose that
\begin{equation*}
\frac{r_5}{r_6}\leq\frac{|a_2|}{|b_2|},\frac{|c_2|}{|d_2|}\leq\frac{r_7}{r_8}
\end{equation*}
where $r_i\in\mathbb{Z}\backslash\{0\}$ for $i=5,6,7,8$ with $|d_2|>r_6$ and $|d_2|>r_8$. Then
\begin{equation*}
\frac{r_5}{r_6}\leq\frac{|a_3|}{|b_3|},\frac{|c_3|}{|d_3|}\leq\frac{r_7}{r_8}.
\end{equation*}
\end{lemma}
\begin{proof}
We will assume that $\frac{c_1b_2}{d_1d_2}<0$. The case of $\frac{c_1b_2}{d_1d_2}>0$ follows similarly.

Since $\frac{c_1b_2}{d_1d_2}<0$, $c_1b_2$ and $d_1d_2$ have opposite signs. Thus $|d_3|=|d_1||d_2|-|c_1||b_2|$ since $|d_1|\geq |c_1|$ and $|d_2|\geq |b_2|$. Similarly, since the determinants of the matrices is $1$, we can argue similarly that $|b_3|=|b_1||d_2|-|a_1||b_2|$, $|c_3|=|d_1||c_2|-|c_1||a_2|$, and $|a_3|=|b_1||c_2|-|a_1||a_2|$. We have
\begin{align*}
r_2|a_3|-r_1|c_3|&=r_2|b_1||c_2|-r_2|a_1||a_2|-r_1|d_1||c_2|+r_1|c_1||a_2|\\
&=(r_2|b_1|-r_1|d_1|)|c_2|-(r_2|a_1|-r_1|c_1|)|a_2|.
\end{align*}
Since the determinants of the matrices is $1$, we have $|b_1||c_1|\geq |a_1||d_1|-1$. Thus we have
\begin{equation*}
\frac{|b_1|}{|d_1|} \geq\frac{|a_1|}{|c_1|}-\frac{1}{|c_1||d_1|}
\text{\ \ \ and\ \ \ }
\frac{|b_1|}{|d_1|}-\frac{r_1}{r_2} \geq\frac{|a_1|}{|c_1|}-\frac{r_1}{r_2}-\frac{1}{|c_1||d_1|}
\end{equation*}
Since $|d_1|\geq|c_1|$ and $|d_1|>r_2$, we have
\begin{equation*}
r_2|b_1|-r_1|d_1|\geq r_2|a_1|-r_1|c_1|-\frac{r_2}{|d_1|}>r_2|a_1|-r_1|c_1|-1.
\end{equation*}
Since $r_2|b_1|-r_1|b_1|,r_2|a_1|-r_1|c_1|\in\mathbb{Z}$, we have $r_2|b_1|-r_1|d_1|\geq r_2|a_1|-r_1|c_1|\geq 0$. Since $|c_2|\geq|a_2|$, we thus have $r_2|a_3|-r_1|c_3|\geq 0$. Thus $\frac{r_1}{r_2}\leq\frac{|a_3|}{|c_3|}$. Next, we have
\begin{align*}
r_3|c_3|-r_4|a_3|&=r_3(|d_1||c_2|-|c_1||a_2|)-r_4(|b_1||c_2|-|a_1||a_2|)\\
&=(r_3|d_1|-r_4|b_1|)|c_2|-(r_3|c_1|-r_4|a_1|)|a_2|.
\end{align*}
Since the determinants of the matrices is $1$, we have $|a_1||d_1|\leq |b_1||c_1|+1$. Thus we have
\begin{equation*}
\frac{|b_1|}{|d_1|} \leq\frac{|a_1|}{|c_1|}+\frac{1}{|c_1||d_1|}
\text{\ \ \ and\ \ \ }
\frac{r_3}{r_4}-\frac{|b_1|}{|d_1|} \geq\frac{r_3}{r_4}-\frac{|a_1|}{|c_1|}-\frac{1}{|c_1||d_1|}.
\end{equation*}
Since $|d_1|\geq|c_1|$ and $|d_1|>r_4$, we have
\begin{equation*}
 r_3|d_1|-r_4|b_1|\geq r_3|c_1|-r_4|a_1|-\frac{r_4}{|d_1|}>r_3|c_1|-r_4|a_1|-1.
\end{equation*}
Since $r_3|d_1|-r_4|b_1|,r_3|c_1|-r_4|a_1|\in\mathbb{Z}$, we have $r_3|d_1|-r_4|b_1|\geq r_3|c_1|-r_4|a_1|\geq 0$. Since $|c_2|\geq |a_2|$, we thus have $r_3|c_3|-r_4|a_3|\geq 0$. Thus $\frac{|a_3|}{|c_3|}\leq\frac{r_3}{r_4}$. The rest of the inequalities follow similarly.
\end{proof}
We prove that the matrices in Example \ref{entriessize} satisfies Notation \ref{P_1P_2} and Theorem \ref{bigthm}.
\begin{proof}
First, $A$ and $B$ consist of integer entries and both matrices have determinate $-1$. We can prove by induction on $j,k\in\mathbb{N}$ that
\begin{equation*}
A^j=
\begin{bmatrix}
F_{j-1} & F_j\\
F_j & F_{j+1}
\end{bmatrix}
\text{\ \ \ and \ \ \ }
B^k=(-1)^k
\begin{bmatrix}
F_{k-1} & -F_k\\
-F_k & F_{k+1}
\end{bmatrix}
\end{equation*}
where $F_k$ is the $k$th Fibonacci number where $F_0=0$, and $F_1=F_2=1$.

Let $P_1$ and $P_2$ be product matrices of $A$s and $B$ satisfying the example. Without loss of generality, it is enough to prove that $P_1$ satisfies the matrix $P_1$ in Notation \ref{P_1P_2}. We prove this by induction on the number of matrices of the form $A^j$ and $B^k$ there are in the product. For the base cases of $P_1=A^k$ and $P_1=B^k$, we have
\begin{equation*}
0\leq\frac{F_{k-1}}{F_k},\frac{F_k}{F_{k+1}}\leq 1,
\end{equation*}
$F_{k+1}\geq 2$, and $F_k\neq 0$.

Suppose the case holds for some product matrix $P$
\begin{equation*}
P=
\begin{bmatrix}
a & b\\
c & d
\end{bmatrix}.
\end{equation*}
First we prove it holds for $PA^j$ where $j\geq 2$. By induction, we have $|d|\geq 2$ and $c\neq 0$. Thus $\frac{cF_{j-1}}{dF_j}\neq 0$.

First assume that $\frac{cF_{j-1}}{dF_j}>0$. By induction, we have all of the inequalities holding in Lemma \ref{positive} with $r_1=0$ and $r_2=r_3=r_4=1$. Lemma \ref{positive} thus gives us all of the desired inequalities holding for $PA^j$ with the observations that $|c||F_j|+|d||F_{j+1}|\geq 2$, $|c|F_{k-1}+|d|F_k\neq 0$, and $|a|F_k+|b|F_{k+1}\neq 0$.

Now assume that $\frac{cF_{j-1}}{dF_j}<0$. By induction, we have all of the inequalities holding in Lemma \ref{positive} with $r_1=r_5=0$ and $r_2=r_3=r_4=r_6=r_7=r_8=1$. Lemma \ref{positive} thus gives us all of the desired inequalities holding for $PA^j$ with the observations that $|d|F_{j+1}-|c|F_j\geq |d|\geq 2$, $d|F_k|-c|F_{k-1}|\neq 0$, and $|d|F_{k+1}-|c|F_k\neq 0$.

The case of $PB^k$ is similar.
\end{proof}
\begin{lemma}\label{ratio1}
Consider a matrix $P$
\begin{equation*}
P=
\begin{bmatrix}
a & b\\
c & d
\end{bmatrix}.
\end{equation*}
where $1\leq |a|<|c|<|d|$, $|a|<|b|<|d|$, and $|\det P|=|ad-bc|=1$. 
\begin{enumerate}
\item Suppose there doesn't exist positive integers $r_1,r_2,r_3,r_4$ with $r_3<r_4$ such that
\begin{equation*}
\frac{r_1}{r_2}\leq\frac{|a|}{|c|},\frac{|b|}{|d|}\leq\frac{r_3}{r_4}
\end{equation*}
and $|d|>r_2,r_4$. Then $|b|=|c|=|d|-1=|a|+1$.
\label{case1}
\item Suppose there doesn't exist positive integers $r_1,r_2,r_3,r_4$ with $r_3<r_4$ such that
\begin{equation*}
\frac{r_1}{r_2}\leq\frac{|a|}{|b|},\frac{|c|}{|d|}\leq\frac{r_3}{r_4}
\end{equation*}
and $|d|>r_2,r_4$. Then $|b|=|c|=|d|-1=|a|+1$.
\label{case2}
\end{enumerate}
\end{lemma}
\begin{proof}
We prove \ref{case1}. Case \ref{case2} follows by taking the transpose of the matrix $P$ and using \ref{case1}.

Suppose $|c|\geq |a|+2$. Then
\begin{equation*}
\frac{1}{|d|-1}\leq\frac{1}{|c|}\leq\frac{|a|}{|c|}\leq\frac{|c|-2}{|c|}<\frac{|c|-1}{|c|}.
\end{equation*}
Also, we have
\begin{equation*}
\frac{|b|}{|d|}\leq\left|\frac{b}{d}-\frac{a}{c}\right|+\frac{|a|}{|c|}
=\frac{1}{|cd|}+\frac{|a|}{|c|}
\leq\frac{1}{|cd|}+\frac{|c|-2}{|c|}
\leq\frac{|c|-2+\frac{1}{|d|}}{|c|}
<\frac{|c|-1}{|c|}
\end{equation*}
with the last inequality following from $|d|\geq 3$. Also, we have
\begin{equation*}
\frac{|b|}{|d|}\geq\frac{2}{|d|}>\frac{1}{|d|-1}
\end{equation*}
since $|b|\geq 2$ and $|d|\geq 3$. Thus letting $r_1=1$, $r_2=|d|-1$, $r_3=|c|-1$, and $r_4=|c|$, we obtain the existence of four positive integers with the properties as stated in the theorem. Thus we may assume that $|c|=|a|+1$. By similar reasoning, if $|d|\geq |b|+2$, then we can see that $r_1=1$, $r_2=|d|-1$, $r_3=|d|-2$, and $r_4=|d|-1$ also satisfies the properties as stated in the theorem. Thus we may also assume that $|d|=|b|+1$.

We can deduce that $|a||c|-|b||d|=\pm 1$ from $|ac-bd|=1$. Thus we have
\begin{equation*}
\pm 1=|a|(|b|+1)-|b|(|a|+1)=|a|-|b|.
\end{equation*}
We know, however, that $|a|<|b|$. We therefore have that $|b|=|a|+1$ and so we must also have that $|b|=|c|$.
\end{proof}
\begin{lemma}\label{samesigns}
Consider a matrix $P$
\begin{equation*}
P=
\begin{bmatrix}
a & b\\
c & d
\end{bmatrix},
\end{equation*}
where $|\det P|=|ad-bc|=1$, $a,b,c,d\in\mathbb{Z}\backslash\{0\}$. We have $\frac{a}{c}>0$ if and only if $\frac{b}{d}>0$.
\end{lemma}
\begin{proof}
It suffices to prove that $\frac{ad}{bc}>0$. Suppose for a contradiction that $\frac{ad}{bc}<0$. Then $ad$ and $bc$ have opposite signs. Also notice that $|ad|\geq 1$ and $|bc|\geq 1$. Then we have $|ad-bc|\geq 2$, a contradiction. The result follows.
\end{proof}
\begin{lemma}\label{power}
Consider a matrix $P$
\begin{equation*}
P=
\begin{bmatrix}
a & b\\
c & d
\end{bmatrix},
\end{equation*}
where $a,b,c,d\in\mathbb{Z}$, $|c|,|d|\geq 2$, $b\neq 0$, $\det(P)=\pm 1$, and
\begin{equation*}
0\leq\frac{|a|}{|b|},\frac{|c|}{|d|},\frac{|a|}{|c|},\frac{|b|}{|d|}\leq 1.
\end{equation*}
For all $i\in\mathbb{N}$, let
\begin{equation*}
P^i=
\begin{bmatrix}
a_i & b_i\\
c_i & d_i
\end{bmatrix}.
\end{equation*}
For all $i\geq 2$, we have $|d_i|-|b_i|\geq|d_{i-1}|-|b_{i-1}|$, $|d_i|-|c_i|\geq|d_{i-1}|-|c_{i-1}|$, $|b_i|-|a_i|\geq(|b|-|a|)(|b_{i-1}|-|a_{i-1}|)$, and $|d_i|>|d_{i-1}|$.
\end{lemma}
\begin{proof}
We can deduce that either $|d_i|=|d||d_{i-1}|-|c||b_{i-1}|$ and $|b_i|=|b||d_{i-1}|-|a||b_{i-1}|$ or $|d_i|=|d||d_{i-1}|+|c||b_{i-1}|$ and $|b_i|=|b||d_{i-1}|+|a||b_{i-1}|$. In the first case, we have
\begin{align*}
|d_i|-|b_i|&=|d||d_{i-1}|-|c||b_{i-1}|-(|b||d_{i-1}|-|a||b_{i-1}|)\\
&=(|d|-|b|)|d_{i-1}|-(|c|-|a|)|b_{i-1}|.
\end{align*}
Since $\det(P)=\pm 1$, we can deduce that $|a||d|\geq |b||c|-1$. Thus, we have $|a||d|-|a||c|\geq |b||c|-|a||c|-1$. Since $\det(P)=\pm 1$, we have $\gcd(|a|,|c|)=1$ and so since $|c|\geq 2$ and $|c|\geq |a|$, we have $|c|>|a|$. Since $|a||d|-|a||c|\geq |b||c|\geq 0$, we have
\begin{equation*}
|d|-|c|\geq |b|-|a|-\frac{1}{|c|}>|b|-|a|-1.
\end{equation*}
It follows that $|d|-|b|\geq |c|-|a|$ since $a,b,c,d\in\mathbb{Z}$. Thus
\begin{equation*}
|d_i|-|b_i|\geq(|c|-|a|)(|d_{i-1}|-|b_{i-1}|)\geq |d_{i-1}|-|b_{i-1}|.
\end{equation*}
The other inequalities follows similarly.
\end{proof}
\begin{lemma}\label{limitsinfinity}
For all $m\in\mathbb{N}$ let $q_m$, $P_m$, $a_m$, $b_m$, $c_m$, and $d_m$ be defined as in Notation \ref{P_1P_2}. We have
\begin{equation*}
\lim_{m\rightarrow\infty}|a_m|=\lim_{m\rightarrow\infty}|b_m|=\lim_{m\rightarrow\infty}|c_m|=\lim_{m\rightarrow\infty}|d_m|=\infty.
\end{equation*}
\end{lemma}
\begin{proof}
By Lemma \ref{positive}, we have for all $m\in\mathbb{N}$ that $\min\{|a_m|,|b_m|,|c_m|,|d_m|\}=|a_m|$, $\max\{|a_m|,|b_m|,|c_m|,|d_m|\}=|d_m|$, and $|d_m|\geq 2$. By induction for all $m\in\mathbb{N}$ we have $\det(P_m)=|a_md_m-b_mc_m|=\pm 1$. Thus for all $m\in\mathbb{N}$, we have $\gcd(|b_m|,|d_m|)=\gcd(|c_m|,|d_m|)=1$ so that for all $m\in\mathbb{N}$ $|c_m|<|d_m|$ and $|b_m|<|d_m|$. Also, for all $m\in\mathbb{N}$, let
\begin{equation*}
P_{m+1}^{q_m}=
\begin{bmatrix}
a_{m+1,q_m} & b_{m+1,q_m}\\
c_{m+1,q_m} & d_{m+1,q_m}
\end{bmatrix}.
\end{equation*}
From Lemma \ref{power}, we can deduce that  $|c_{m,q}|<|d_{m,q}|$ and $|b_{m,q}|<|d_{m,q}|$. Thus, for all $m\in\mathbb{N}$, we have by Lemma \ref{power} 
\begin{equation*}
|d_{m+2}|\geq |d_{m+1,q_m}||d_{m}|-|c_{m+1,q_m}||b_{m}|>|d_m|(|d_{m}|-|c_{m}|)\geq |d_m|.
\end{equation*}
It follows that $\lim_{m\rightarrow\infty}|d_m|=\infty$. Also, we have
\begin{equation*}
|b_{m+2}|\geq |b_{m+1,q_m}||d_{m}|-|a_{m+1,q_m}||b_{m}|\geq|b_{m+1,q_m}|(|d_{m}|-|b_{m}|)\geq |b_{m+1,q_m}|.
\end{equation*}
By similar reasoning, we have $|b_{m+1,q_m}|\geq |b_{m+1}|$ and so for all $m\in\mathbb{N}$ we have $|b_{m+2}|\geq |b_{m+1}|$ with equality only if $|a_{m+1,q_m}|=|b_{m+1,q_m}|=1$ and $|d_{m}|=|b_{m}|+1$. If $|b_m|$ is bounded for all $m\in\mathbb{N}$, then for sufficiently large $m$, we have $|d_{m}|=|b_{m}|+1$ and so $|d_m|$ is bounded, a contradiction. Thus $|b_m|$ isn't bounded and we have $\lim_{m\rightarrow\infty}|b_m|=\infty$.

Finally, choose $M\in\mathbb{N}$ such that for all $m\geq M$ $|b_m|\geq 2$. Then for all $m\geq M$, we have $|b_m|>|a_m|$. By Lemma \ref{power}, we have $|b_{m+1,q_m}|-|a_{m+1,q_m}|\geq(|b_{m+1}|-|a_{m+1}|)^{q_m}\geq 1$. So for all $m\geq M$, we have
\begin{align*}
|a_{m+2}|&=|a_{m+1,q_m}a_{m}+b_{m+1,q_m}c_{m}|\\
&\geq |b_{m+1,q_m}||c_{m}|-|a_{m+1,q_m}||a_{m}|\\
&>(|b_{m+1,q_m}|-|a_{m+1,q_m}|)|a_{m}|\\
&\geq |a_m|.
\end{align*}
Thus $\lim_{m\rightarrow\infty}|a_m|=\infty$. Since $|a_m|<|c_m|$ for all $m\in\mathbb{N}$, we also have $\lim_{m\rightarrow\infty}|c_m|=\infty$.
\end{proof}
Letting $a_m$, $b_m$, $c_m$, and $d_m$ be as defined in Notation \ref{P_1P_2}, Lemma \ref{limitsinfinity} implies that for sufficiently large $m$, we have $|d_m|>|b_m|>|a_m|>0$.
\begin{remark}
In Notation \ref{P_1P_2}, since $P_3=(P_2)^{q_1}P_1$, $P_2$ is $P_3$ truncated after a certain point. Thus, by reindexing the matrices $P_i,{i\in\mathbb{N}}$, we will assume for the rest of the paper that $P_1$ is $P_2$ truncated after a certain point.
\end{remark}
\begin{lemma}\label{Q_nproductP_n}
Let $q_m$, $P_m$, and $Q_n$ be as defined in Notation \ref{P_1P_2}. Then there exists uniquely $2\leq m_1<m_2<...<m_l$ and $n_1,\ldots,n_l\in\mathbb{N}$ with the following properties:
\begin{enumerate}
\item for all $1\leq i\leq l$, we have $n_i\leq q_{m_i-1}$,
\item for all $2\leq i\leq l$, if $n_i=q_{m_i-1}$, then $m_{i-1}+2\leq m_i$,
\item $Q_n=(P_{m_l})^{n_l}(P_{m_{l-1}})^{n_{l-1}}...(P_{m_1})^{n_1}M_n$ where the matrix $M_n$ is a product of the string of $A_1$s, $A_2$s, $\ldots$, and $A_v$s in $P_2$ truncated after a certain point unless $m_1=2$ and $n_1=q_1$ in which case $M_n$ is a product of the string of $A_1$s, $A_2$s, $\ldots$, and $A_v$s in $P_1$ truncated after a certain point.
\end{enumerate}
\end{lemma}
\begin{proof}
We prove by strong induction on $n\in\mathbb{N}$. If $n<k_2$, then we have $Q_n=M_n$. Suppose $n\geq k_2$ and that the lemma holds true for all values less than $n$. Choose $m\in\mathbb{N}$ such that $k_{m+1}>n\geq k_m$ where $m\geq 2$. Then $P_m$ is $Q_n$ truncated after a certain point and $Q_n$ is $P_{m+1}$ truncated after a certain point. Since $P_{m+1}=(P_m)^{q_{m-1}}P_{m-1}$, it follows that $Q_n=(P_m)^iR$ where $i$ and $R$ satisfy the following:
\begin{enumerate}
\item $i\leq q_{m-1}$
\item $R$ is a product of the string of $A_1$s, $A_2$s, $\ldots$, and $A_v$s in $P_m$ truncated after a certain point
\item if $i=q_{m-1}$ and $m\geq 3$, then $R$ is a product of the string of $A_1$s, $A_2$s, $\ldots$, and $A_v$s in $P_{m-1}$ truncated after a certain point.
\end{enumerate}
Notice that if we replace $m$ by any other positive integer, say $m'$ and have $Q_n=\left(P_{m'}\right)^iR$ instead, then it follows that $m'<m$. But then if $i\leq q_{m-1}$, then we have that $P_{m'}$ is $R$ truncated after a certain point. Thus any expression for $Q_n$ in (3) in the lemma must begin with $m_l=m$. Similarly the value of $i$ satisfying the above must be unique.

If $m=2$ and $i<q_1$, then $R$ is $P_2$ truncated after a certain point and $R=M_n$ and the result follows. If $m=2$ and $i=q_1$, then $R$ is $P_1$ truncated after a certain point and $R=M_n$ and again the result follows. So assume that $m>2$. Then $R=Q_{n-ik_m}$ and so $Q_n=P_m^iQ_{n-ik_m}$ with $n-ik_m<k_m$ and if $i=q_{n-1}$, then $n-ik_m<k_{m-1}$. By induction, the result follows.
\end{proof}
We divide into two cases. Case $1$ assumes that for sufficiently large $m$, we have $|b_m|=|c_m|=|d_m|-1=|a_m|+1$. Case $2$ deals with all other cases.
\section{The Linear Growth Case}\label{sec3}
For this case, we may assume without loss of generality that $|b_m|=|c_m|=|d_m|-1=|a_m|+1$ for all $m\in\mathbb{N}$. By Lemma \ref{limitsinfinity}, we may also assume without loss of generality that $|a_m|\geq 1$ for all $m\in\mathbb{N}$.
\begin{lemma}\label{nklinear}
Let $P_m$ and $Q_n$ be as defined in Notation \ref{P_1P_2}. Let $n_1,n_2,\ldots,n_j,\ldots,$ be the list of natural numbers such that for each $n_j$, there exists $2\leq m_1<m_2<...<m_l$ and $n_{j,1},\ldots,n_{j,l}\in\mathbb{N}$ with the following properties:
\begin{enumerate}
\item for all $1\leq i\leq l$, we have $n_{j,i}\leq q_{m_i-1}$
\item for all $2\leq i\leq l$, if $n_{j,i}=q_{m_i-1}$, then $m_{i-1}+2\leq m_i$
\item $Q_{n_j}=(P_{m_l})^{n_{j,l}}(P_{m_{l-1}})^{n_{j,l-1}}...(P_{m_1})^{n_{j,1}}$.
\end{enumerate}
We have
\begin{equation*}
|g_{n_j}|\leq n_j\max\{|c_1|,|c_2|\}
\end{equation*}
for all $j\in\mathbb{N}$.
\end{lemma}
\begin{proof}
We prove by induction on $k\in\mathbb{N}$. First we observe that Lemma \ref{positive} implies that $|e_{n_k}|\leq|f_{n_k}|\leq |h_{n_k}|$ and $|e_{n_k}|\leq|g_{n_k}|\leq |h_{n_k}|$ for all $k\in\mathbb{N}$. Let $k\in\mathbb{N}$. If there exists $r_1,r_2,r_3,r_4\in\mathbb{N}$ with $r_3<r_4$ such that
\begin{equation*}
\frac{r_1}{r_2}\leq\frac{|e_{n_k}|}{|g_{n_k}|},\frac{|f_{n_k}|}{|h_{n_k}|}\leq\frac{r_3}{r_4}
\end{equation*}
and $|h_{n_k}|>r_2,r_4$, then we can use Lemma \ref{positive} to deduce that for all sufficiently large $m\in\mathbb{N}$, we have
\begin{equation*}
\frac{r_1}{r_2}\leq\frac{|a_m|}{|c_m|},\frac{|b_m|}{|d_m|}\leq\frac{r_3}{r_4}.
\end{equation*}
But this cannot be because
\begin{equation*}
\lim_{m\rightarrow\infty}\frac{|a_m|}{|c_m|}=\lim_{m\rightarrow\infty}\frac{|b_m|}{|d_m|}=1.
\end{equation*}
We can therefore see with Lemma \ref{ratio1} that for all $k\in\mathbb{N}$, we have $|f_{k_m}|=|g_{k_m}|=|h_{k_m}|-1=|e_{k_m}|+1$. Suppose the desired inequality holds for $k$. We will prove it also holds for $k+1$. Notice that Lemma \ref{Q_nproductP_n} implies that for all $n\in\mathbb{N}$ $Q_n$ is a product of the matrices $P_1$ and $P_2$ if and only if $n$ is in the sequence $(n_j)_j$. It follows that $Q_{n_{k+1}}=Q_{n_k}P_1$ or  $Q_{n_{k+1}}=Q_{n_k}P_2$. Then we have that either $|g_{n_{k+1}}|=|h_{n_k}||c_i|+|g_{n_k}||a_i|$ or $|g_{n_{k+1}}|=|h_{n_k}||c_i|-|g_{n_k}||a_i|$ where $i=1$ or $2$. Suppose the first equality holds. Then we also have $|e_{n_{k+1}}|=|e_{n_k}||a_i|+|f_{n_k}||c_i|$ so that
\begin{align*}
|g_{n_{k+1}}|-|e_{n_{k+1}}|&=|h_{n_k}||c_i|+|g_{n_k}||a_i|-|e_{n_k}||a_i|-|f_{n_k}||c_i|\\
&=(|h_{n_k}|-|f_{n_k}|)|c_i|+(|g_{n_k}|-|e_{n_k}|)|a_i|\\
&=|c_i|+|a_i|\\
&>1,
\end{align*}
a contradiction. Thus we have
\begin{align*}
 |g_{n_{k+1}}|&=|h_{n_k}||c_i|-|g_{n_k}||a_i|\\
&=(|g_{n_k}|+1)|c_i|-|g_{n_k}||c_i|+|g_{n_k}|\\
&=|g_{n_k}|+|c_i|\\
&\leq n_k\max\{|c_1|,|c_2|\}+|c_i|\\
&\leq(n_k+1)\max\{|c_1|,|c_2|\}\\
&\leq(n_{k+1})\max\{|c_1|,|c_2|\}.
\end{align*}
\end{proof}
\begin{proposition}\label{entrieslimit1}
Let $e_n$, $f_n$, $g_n$, and $h_n$ be as defined in Notation \ref{P_1P_2}. Then there exists $D>0$ such that for all $n\in\mathbb{N}$, we have
\begin{equation*}
|e_n|,|f_n|,|g_n|,|h_n|<Dn.
\end{equation*}
\end{proposition}
\begin{proof}
Consider all product matrices constructed as the string of $A$s and $B$s in $P_1$ truncated after a certain point and in $P_2$ truncated after a certain point. Let $M$ be the largest entry in absolute value of all such matrices. Let $n\in\mathbb{N}$. By Lemma \ref{Q_nproductP_n}, we have there exists $2\leq m_1<m_2<...<m_l$ and $
n_1,\ldots,n_l\in\mathbb{N}$ with the following properties:
\begin{enumerate}
\item for all $1\leq i\leq l$, we have $n_i\leq q_{m_i-1}$
\item for all $2\leq i\leq l$, if $n_i=q_{m_i-1}$, then $m_{i-1}+2\leq m_i$.
\item  $Q_n=(P_{m_l})^{n_l}(P_{m_{l-1}})^{n_{l-1}}...(P_{m_1})^{n_1}M_n$ where the matrix $M_n$ is a product of the string of $A_1$s, $A_2$s, $\ldots$, and $A_v$s in either $P_1$ or $P_2$ truncated after a certain point.
\end{enumerate}

Let
\begin{equation*}
Q_{n'}=(P_{m_l})^{n_l}(P_{m_{l-1}})^{n_{l-1}}...(P_{m_1})^{n_1}
\end{equation*}
so that
\begin{equation*}
Q_n=Q_{n'}M_n.
\end{equation*}
Note that $n'<n+k_2$. By Lemma \ref{nklinear}, we have $|g_{n'}|\leq n'\max\{|c_1|,|c_2|\}$. Let $C=\max\{|c_1|,|c_2|\}$. Then we have
\begin{align*}
\max\{|e_n|,|f_n|,|g_n|,|h_n|\}& \leq 2M(|g_{n'}|+1)\leq 2M(Cn'+1)<2M(C(n+k_2)+1)\\ & \leq(2MC+2MCk_2+2M)n.
\end{align*}
Since none of $M,C$, or $k_2$ depend on $n$, letting $D=2MC+2MCk_2+2M$, we obtain our result.
\end{proof}
\begin{proof}[Proof of Theorem \ref{bigthm} for Case $1$]
From Proposition \ref{entrieslimit1}, we get there exists $C>0$ such that for all $n\in\mathbb{N}$
\begin{equation*}
|G_n|<Cn.
\end{equation*}
\end{proof}
\section{The Exponential Growth Case}\label{sec4}
Case $2$ covers all other cases. By Lemma \ref{ratio1}, there exists $m\in\mathbb{N}$, $m\geq 2$, with the following properties. $a_m,b_m,c_m,d_m$ are all nonzero and there exists positive integers $r_1,r_2,r_3,r_4,r_5,r_6,r_7,r_8$ such that 
\begin{equation*}
\frac{r_1}{r_2}\leq\frac{|a_m|}{|c_m|},\frac{|b_m|}{|d_m|}\leq\frac{r_3}{r_4},
\text{\ \ \ and\ \ \ }
\frac{r_5}{r_6}\leq\frac{|a_m|}{|b_m|},\frac{|c_m|}{|d_m|}\leq\frac{r_7}{r_8},
\end{equation*} 
with $r_1<r_2$, $r_3<r_4$, $r_5<r_6$, $r_7<r_8$, and $|d_m|>r_2,r_4,r_6,r_8$. Without loss of generality, we may assume that the inequalities hold for $m=2$. Also, by taking the minimum of $\frac{r_1}{r_2}$ and $\frac{r_5}{r_6}$ and the maximum of $\frac{r_3}{r_4}$ and $\frac{r_7}{r_8}$, we can say there exists positive integers $r_1,r_2,r_3,r_4$ such that
\begin{equation*}
\frac{r_1}{r_2}\leq\frac{|a_2|}{|c_2|},\frac{|b_2|}{|d_2|},\frac{|a_2|}{|b_2|},\frac{|c_2|}{|d_2|}\leq\frac{r_3}{r_4}
\end{equation*} 
with $r_1<r_2$, $r_3<r_4$, and $|d_2|>r_2,r_4$.
\begin{lemma}\label{ratios1}
Let $a_m$, $b_m$, $c_m$, and $d_m$ be as defined in Notation \ref{P_1P_2}. For all $m\geq 2$, we have
\begin{equation*}
\frac{r_1}{r_2}\leq\frac{|a_m|}{|c_m|},\frac{|b_m|}{|d_m|}\leq\frac{r_3}{r_4}
\end{equation*}
and $|d_m|>r_2,r_4$.
\end{lemma}
\begin{proof}
We prove our result by induction on $m\in\mathbb{N}$. We have already established it for $m=2$. Suppose the case holds for some $m$ where $m\geq 2$. We prove it holds for $P_{m+1}=(P_m)^{q_{m-1}}P_{m-1}$.\par
First assume that $\frac{c_{m,q_{m-1}}b_{m-1}}{d_{m,q_{m-1}}d_{m-1}}>0$. By Lemma \ref{power}, we have $|d_{m,q_{m+1}}|\geq |c_{m,q_{m+1}}|+1$. By Lemma \ref{positive}, we have the desired inequalities holding for $P_{m+1}$ with the observation by Lemma \ref{power} that 
\begin{align*}
|d_{m+1}|&\geq |d_{m,q_{m-1}}||d_{m-1}|-|c_{m,q_{m-1}}||b_{m-1}|\\
&\geq |d_{m,q_{m-1}}|(|b_{m-1}|+1)-|c_{m,q_{m-1}}||b_{m-1}|\\
&\geq(|d_{m,q_{m-1}}|-|c_{m,q_{m-1}}|)|b_{m-1}|+|d_m,q_{m-1}|\\
&>|d_m,q_{m-1}|\\
&>|d_m|\\
&>r_2,r_4.
\end{align*}
Now assume that $\frac{c_{m,q_{m-1}}b_{m-1}}{d_{m,q_{m-1}}d_{m-1}}<0$. By induction, we have all of the inequalities holding in Lemma \ref{positive}. Lemma \ref{positive} thus gives us all of the desired inequalities holding for $P_{m+1}$ again with the observation that $|d_{m+1}|>|d_m|>r_2,r_4$.
\end{proof}
Also, with the help of Lemma \ref{positive}, we can obtain the following.
\begin{lemma}\label{ratios3}
Let $a_m$, $b_m$, $c_m$, and $d_m$ be as defined in Notation \ref{P_1P_2}. For all even $m\geq 2$, we have
\begin{equation*}
\frac{r_1}{r_2}\leq\frac{|a_m|}{|b_m|},\frac{|c_m|}{|d_m|}\leq\frac{r_3}{r_4}
\end{equation*}
and $|d_m|>r_2,r_4$.
\end{lemma}
By Lemmas \ref{ratio1}, we can also obtain that there exists positive integers $r_9,r_{10},r_{11},r_{12}$ such that
\begin{equation*}
\frac{r_9}{r_{10}}\leq\frac{|a_3|}{|b_3|},\frac{|c_3|}{|d_3|}\leq\frac{r_{11}}{r_{12}}
\end{equation*} 
with $r_9<r_{10}$, $r_{11}<r_{12}$, and $|d_2|>r_{10},r_{12}$ and so also with the help of Lemma \ref{positive}, we can obtain the following.
\begin{lemma}\label{ratios2}
Let $a_m$, $b_m$, $c_m$, and $d_m$ be as defined in Notation \ref{P_1P_2}. For all odd $m\geq 3$, we have
\begin{equation*}
\frac{r_9}{r_{10}}\leq\frac{|a_m|}{|b_m|},\frac{|c_m|}{|d_m|}\leq\frac{r_{11}}{r_{12}}
\end{equation*}
and $|d_m|>r_2,r_4$.
\end{lemma}
\begin{remark}
Without loss of generality, we will assume that $r_9=r_1$, $r_{10}=r_2$, $r_{11}=r_3$, and $r_{12}=r_4$ for the rest of this section.
\end{remark}
\begin{lemma}\label{qratios}
Let $a_m$, $b_m$, $c_m$, and $d_m$ be as defined in Notation \ref{P_1P_2}. For all $m\geq 2$ and $i\in\mathbb{N}$, let
\begin{equation*}
P_m^i:=\begin{bmatrix}
a_{m,i} & b_{m,i}\\
c_{m,i} & d_{m,i}
\end{bmatrix}.
\end{equation*}
Then we have
\begin{equation*}
\frac{r_1}{r_2}\leq\frac{|a_{m,i}|}{|c_{m,i}|},\frac{|b_{m,i}|}{|d_{m,i}|},\frac{|a_{m,i}|}{|b_{m,i}|},\frac{|c_{m,i}|}{|d_{m,i}|}\leq\frac{r_3}{r_4}.
\end{equation*}
\end{lemma}
\begin{remark}
For the rest of the section, we will let $t_1=\frac{(r_4-r_3)}{r_3r_4}$ and $t_2=\frac{(r_2r_4+r_1r_3)}{r_1r_4}$.
\end{remark}
\begin{lemma}\label{c_mrelations}
Let $c_m$ be as defined in Notation \ref{P_1P_2}. We have
\begin{equation*}
(t_1|c_{m-1}|)^{q_{m-2}}|c_{m-2}|\leq|c_m|\leq(t_2|c_{m-1}|)^{q_{m-2}}|c_{m-2}|
\end{equation*}
for all $m\geq 4$.
\end{lemma}
\begin{proof}
Let $m\geq 4$. By Lemmas \ref{ratios3}, \ref{ratios2}, and \ref{qratios}, we have
\begin{align*}
|c_m| &\geq |d_{m-1,q_{m-2}}||c_{m-2}|- |c_{m-1,q_{m-2}}||a_{m-2}|\\
&\geq\frac{r_4}{r_3}|c_{m-1,q_{m-2}}||c_{m-2}|-\frac{r_3}{r_4}|c_{m-1,q_{m-2}}||c_{m-2}|\\
&=\frac{(r_4-r_3)}{r_3r_4}|c_{m-1,q_{m-2}}||c_{m-2}|.
\end{align*}
Through induction on $q_{m-2}$, we can similarly derive that $|c_{m-1,q_{m-2}}|\geq t_1^{q_{m-2}-1}|c_{m-1}|^{q_{m-2}}$. Also, we have
\begin{align*}
|c_m| &\leq |d_{m-1,q_{m-2}}||c_{m-2}|+|c_{m-1,q_{m-2}}||a_{m-2}|\\
&\leq\frac{r_2}{r_1}|c_{m-1,q_{m-2}}||c_{m-2}|+\frac{r_3}{r_4}|c_{m-1,q_{m-2}}||c_{m-2}|\\
&=\frac{(r_2r_4+r_1r_3)}{r_1r_4}|c_{m-1,q_{m-2}}||c_{m-2}|.
\end{align*}
Again, through induction on $q_{m-2}$, we can similarly derive that $|c_{m-1,q_{m-2}}|\leq t_2^{q_{m-2}-1}|c_{m-1}|^{q_{m-2}}$. Thus we have our result.
\end{proof}
\begin{proposition}\label{limc_mexists}
Let $q_m$, $c_m$, and $k_m$ be as defined in Notation \ref{P_1P_2}. We have
\begin{equation*}
\lim_{m\rightarrow\infty}|c_m|^{1/k_m}
\end{equation*}
exists, is finite, and is greater than $1$.
\end{proposition}
\begin{proof}
It suffices to show that
\begin{equation*}
\lim_{m\rightarrow\infty}\frac{\log|c_m|}{k_m}
\end{equation*}
exists, is finite, and is positive. Let $u_m=\log |c_m|$ and $s_m:=\frac{u_m}{k_m}$. By Lemma \ref{c_mrelations}, we have that $q_{m-1}(u_m+\log t_1)+u_{m-1}\leq u_{m+1}\leq q_{m-1}(u_m+\log t_2)+u_{m-1}$ for all $m\in\mathbb{N}$, $m\geq 3$. Let $m\in\mathbb{N}$, $m\geq 3$. We have
\begin{align}
s_{m+1}-s_m&=\frac{u_m}{k_{m+1}}\left(\frac{u_{m+1}}{u_m}-\frac{k_{m+1}}{k_m}\right)\nonumber\\
&\leq\frac{u_m}{k_{m+1}}\left(\frac{q_{m-1}(u_m+\log t_2)+u_{m-1}}{u_m}-\frac{k_{m+1}}{k_m}\right)\nonumber\\
&=\frac{u_m}{k_{m+1}}\left(\frac{u_{m-1}+q_{m-1}\log t_2}{u_m}-\frac{k_{m-1}}{k_m}\right)\nonumber\\
&=\frac{u_{m-1}}{k_{m+1}}-\frac{u_mk_{m-1}}{k_{m+1}k_m}+\frac{q_{m-1}\log t_2}{k_{m+1}}\nonumber\\
&=\frac{\left(u_{m-1}-\dfrac{u_mk_{m-1}}{k_m}\right)}{k_{m+1}}+\frac{q_{m-1}\log t_2}{k_{m+1}}\nonumber\\
&=\frac{k_{m-1}(s_{m-1}-s_m)}{k_{m+1}}+\frac{q_{m-1}\log t_2}{k_{m+1}}.\label{t2ineq}
\end{align}
Similarly, we have
\begin{equation}
s_m-s_{m+1}\leq\frac{k_{m-1}(s_{m-1}-s_m)}{k_{m+1}}+\frac{q_{m-1}\log t_1}{k_{m+1}}.\label{t1ineq}
\end{equation}
Therefore
\begin{equation*}
|s_m-s_{m+1}|\leq\frac{k_{m-1}|s_m-s_{m-1}|}{k_{m+1}}+\frac{q_{m-1}\log t}{k_{m+1}},
\end{equation*}
where $t=\max\{t_2,t_1^{-1}\}$. Note that
\begin{equation*}
\frac{k_{m+1}}{k_{m-1}}=\frac{k_m+q_{m+1}k_{m-1}}{k_{m-1}}>1+q_{m+1}\geq 2.
\end{equation*}
Thus
\begin{equation*}
|s_{m+1}-s_m|\le\frac{|s_m-s_{m-1}|}{2}+\frac{\log t}{k_m}.
\end{equation*}
Consider the Fibonacci sequence $F_1=1, F_2=1$, and $F_m=F_{m-1}+F_{m-2}$ for all $m\geq 3$. Then we have
\begin{equation}
k_m\geq F_m=\left\lfloor\frac{\varphi^m}{\sqrt{5}}\right\rfloor\geq\frac{\varphi^m}{5}.\label{kngrowth}
\end{equation}
Thus 
\begin{equation*}
|s_{m+1}-s_m|\le\frac{|s_m-s_{m-1}|}{2}+\frac{5\log t}{\varphi^m}.
\end{equation*}
We can prove by induction on $l\in\mathbb{N}$ that for all $m\geq 3$ and $l\geq 1$, we have
\begin{align*}
|s_{m+l}-s_{m+l-1}|&\leq\frac{|s_m-s_{m-1}|}{2^l}+\left(\frac{1}{2^{l-1}\cdot\varphi^m}+\frac{1}{2^{l-2}\varphi^{m+1}}+\ldots+\frac{1}{\varphi^{m+l-1}}\right)\log t\\
&<\frac{|s_m-s_{m-1}|}{2^l}+\frac{1}{\varphi^{m+l-1}}\left(1+\frac{\varphi}{2}+\frac{\varphi^2}{4}+\ldots\right)\\
&=\frac{|s_m-s_{m-1}|}{2^l}+\frac{1}{\varphi^{m+l-1}}\left(\frac{2}{2-\varphi}\right).
\end{align*}
By a geometric series argument, the limit exists and is finite. It remains to show the limit is positive. By Lemma \ref{c_mrelations}, we have $q_{m-2}(\log t_1+\log|c_{m-1}|)+\log|c_{m-2}|\leq\log|c_m|$ for all $m\geq 3$. We thus have
\begin{equation*}
q_{m-2}(\log|c_{m-1}|+\log t_1)+(\log|c_{m-2}|+\log t_1)\leq\log|c_m|+\log t_1
\end{equation*}
for all $m\geq 3$. By Lemma \ref{limitsinfinity}, we have $\lim_{m\rightarrow\infty}\log|c_m|=\infty$. We can therefore deduce that there exists $C_2>0$ such that for all sufficiently large $m\in\mathbb{N}$, we have $\log|c_m|+\log t_1>C_2k_m$. It follows that the limit is positive.
\end{proof}
\begin{remark}
Let $q_m$, $c_m$, and $q_m$ be as in Proposition \ref{limc_mexists}. Let $L:=\lim_{m\rightarrow\infty}|c_m|^{1/k_m}$.
\end{remark}
\begin{lemma}\label{prodrelc_m}
Let $P_m$, $c_m$, $Q_n$, $e_n$, $f_n$, $g_n$, and $h_n$ be as defined in Notation \ref{P_1P_2}. Let $n\in\mathbb{N}$ such that there exists 
$2\leq m_1<m_2<...<m_l$ and $n_1,\ldots,n_l\in\mathbb{N}$ with the following properties:
\begin{enumerate}
\item for all $1\leq i\leq l$, we have $n_i\leq q_{m_i-1}$
\item for all $2\leq i\leq l$, if $n_i=q_{m_i-1}$, then $m_{i-1}+2\leq m_i$.
\item $Q_n=(P_{m_l})^{n_l}(P_{m_{l-1}})^{n_{l-1}}...(P_{m_1})^{n_1}$
\end{enumerate}
Then
\begin{equation*}
\frac{r_1}{r_2}\leq\frac{|e_n|}{|g_n|},\frac{|f_n|}{|h_n|},\frac{|e_n|}{|f_n|},\frac{|g_n|}{|h_n|}\leq\frac{r_3}{r_4}.
\end{equation*} 
Also, we have
\begin{equation*}
t_1^{l+n_1+\ldots+n_l}|c_{m_l}|^{n_l}|c_{m_{l-1}}|^{n_{l-1}}...|c_{m_1}|^{n_1}
\leq|g_n|
\leq t_2^{l+n_1+\ldots+n_l}|c_{m_l}|^{n_l}|c_{m_{l-1}}|^{n_{l-1}}...|c_{m_1}|^{n_1}.
\end{equation*}
\end{lemma}
\begin{proof}
The first pair of inequalities follows by similar reasoning as in the proof of Lemma \ref{positive}. The second pair of inequalities can by proved by induction on $l\in\mathbb{N}$ with the base case and induction step proved as in the proof of Lemma \ref{c_mrelations}.
\end{proof}
\begin{proposition}\label{rootgnj}
Let $P_m$, $k_m$, $Q_n$, and $g_n$ be as defined in Notation \ref{P_1P_2}. Let $n_1,n_2,\ldots,n_j,\ldots,$ be the list of natural numbers such that for each $n_j$, there exists $2\leq m_1<m_2<...<m_l$ and $n_{j,1},\ldots,n_{j,l}\in\mathbb{N}$ with the following properties:
\begin{enumerate}
\item for all $1\leq i\leq l$, we have $n_{j,i}\leq q_{m_i-1}$
\item for all $2\leq i\leq l$, if $n_{j,i}=q_{m_i-1}$, then $m_{i-1}+2\leq m_i$
\item $Q_{n_j}=(P_{m_l})^{n_{j,l}}(P_{m_{l-1}})^{n_{j,l-1}}...(P_{m_1})^{j,n_1}$.
\end{enumerate}

We have
\begin{equation*}
\lim_{j\rightarrow\infty}|g_{n_j}|^{1/n_j}=L.
\end{equation*}
\end{proposition}
\begin{proof}
It suffices to prove that
\begin{equation*}
\lim_{j\rightarrow\infty}\frac{\log|g_{n_j}|}{n_j}=\log L.
\end{equation*}
We have
\begin{equation*}
\lim_{m\rightarrow\infty}\frac{\log|c_m|}{k_m}=\log L.
\end{equation*}
Let $\epsilon>0$. Pick $\frac{\log L}{2}>\delta_1>0$ and $\delta_2,\delta_3,\delta_4>0$ such that $(\log L+\delta_1)(1+\delta_2)<\log L+\epsilon$ and $\frac{(1-\delta_3)(\log L-\delta_1)}{(1+\delta_4)}>\log L-\epsilon$. Choose $M\in\mathbb{N}$ such that for all $m\geq M$, we have
\begin{equation}
\left|\frac{\log|c_m|}{k_m}-\log L\right|<\delta_1\label{epsiloncond1},
\end{equation}
\begin{equation}
\frac{10}{\varphi^{M-1}(\varphi-1)\log L}+\frac{2}{k_M\log L}<\frac{\delta_2}{2}\label{epsiloncond5},
\end{equation}
and
\begin{equation}
\frac{10\log(t_1^{-1})}{\varphi^{M-1}(\varphi-1)\log L}+\frac{2\log(t_1^{-1})}{k_M\log L}<\frac{\delta_3}{2}\label{epsiloncond6}.
\end{equation}
By \eqref{kngrowth}, we have $\lim_{m\rightarrow\infty}\frac{m}{k_m}=0$. Thus we can choose $N>M$ such that for all $m\geq N$, we have
\begin{equation}
\frac{2m\widehat{q_{M-1}}}{k_m\log L}+\frac{3\widehat{q_{M-1}}Mk_M}{k_m}<\frac{\delta_2}{2}\label{epsiloncond2},
\end{equation}
\begin{equation}
\frac{2m\widehat{q_{M-1}}\log(t_1^{-1})}{k_m\log L}<\frac{\delta_3}{2}\label{epsiloncond3},
\end{equation}
and
\begin{equation}
\frac{\widehat{q_{M-1}}Mk_M}{k_{m}}<\delta_4\label{epsiloncond4}
\end{equation}
where
\begin{equation*}
\widehat{q_{M-1}}:=\max\{q_1,q_1,\ldots,q_{M-1}\}.
\end{equation*}
Let $n_j\geq k_{N}$. Then there exists $2\leq m_1<m_2<...<m_l$ and $n_{j,1},\ldots,n_{j,l}\in\mathbb{N}$ with the following properties:
\begin{enumerate}
\item for all $1\leq i\leq l$, we have $n_{j,i}\leq q_{m_i-1}$
\item for all $2\leq i\leq l$, if $n_{j,i}=q_{m_i-1}$, then $m_{i-1}+2\leq m_i$
\item  $Q_n=(P_{m_l})^{n_{j,l}}(P_{m_{l-1}})^{n_{j,l-1}}...(P_{m_1})^{j,n_1}$.
\end{enumerate}

By Lemma \ref{prodrelc_m}, we have
\begin{align}
& (l+n_{j,1}+\ldots+n_{j,l})\log(t_1)+n_{j,l}\log|c_{m_l}|+...+n_{j,1}\log|c_{m_1}|\nonumber\\
& \qquad \qquad \leq\log|g_{n_j}|
\leq(l+n_{j,1}+\ldots+n_{j,l})\log(t_2)+n_{j,l}\log|c_{m_l}|+...+n_{j,1}\log|c_{m_1}|\label{gbounds}.
\end{align}
Pick $1\leq y\leq l$ such that $m_y\geq M>m_{y-1}$. Thus $y<M$. By \eqref{epsiloncond1}, we have
\begin{equation}
\log L-\delta_1<\frac{n_{j,l}\log|c_{m_l}|+...+n_{j,y}\log|c_{m_y}|}{n_{j,l}k_{m_l}+...+n_{j,y}k_{m_y}}<\log L+\delta_1\label{epsilon1bounds}.
\end{equation}
Also observe the following.
\begin{align}
\frac{l+n_{j,1}+\ldots+n_{j,l-1}}{\log|c_{m_l}|}&<\frac{(l+q_{m_1-1}+\ldots+q_{m_{l-1}-1})}{\log|c_{m_l}|}\nonumber\\
&<\frac{(l+q_{m_1-1}+\ldots+q_{m_{y-1}-1})}{\log|c_{m_l}|}+\frac{q_{m_y-1}+\ldots+q_{m_{l-1}-1}}{\log |c_{m_l}|}\nonumber\\
&<\frac{l\widehat{q_{M-1}}}{\log |c_{m_l}|}+\frac{\dfrac{k_{m_y+1}}{k_{m_y}}+\ldots+\dfrac{k_{m_{l-1}+1}}{k_{m_l-1}}}{\log |c_{m_l}|}\nonumber\\
&<\frac{l\widehat{q_{M-1}}}{k_{m_l}(\log L-\delta_1)}+\frac{\dfrac{k_{m_y+1}}{k_{m_y}}+\ldots+\dfrac{k_{m_{l-1}+1}}{k_{m_l-1}}}{k_{m_l}(\log L-\delta_1)}\nonumber\\
&<\frac{l\widehat{q_{M-1}}}{k_{m_l}(\log L-\delta_1)}+\frac{1}{(\log L-\delta_1)}\sum_{j=m_y}^{\infty}\frac{1}{k_j}\nonumber\\
&<\frac{l\widehat{q_{M-1}}}{k_{m_l}(\log L-\delta_1)}+\frac{1}{(\log L-\delta_1)}\sum_{j=m_y}^{\infty}\frac{5}{\varphi^y}\nonumber\\
&=\frac{l\widehat{q_{M-1}}}{k_{m_l}(\log L-\delta_1)}+\frac{5}{\varphi^{m_y-1}(\varphi-1)(\log L-\delta_1)}\nonumber\\
&<\frac{2m_l\widehat{q_{M-1}}}{k_{m_l}\log L}+\frac{10}{\varphi^{M-1}(\varphi-1)\log L}.\label{nlog}
\end{align}
Thus, by \eqref{epsiloncond1}, \eqref{epsiloncond5}, and \eqref{epsiloncond2}, we have
\begin{align}
 &\frac{(l+n_{j,1}+\ldots+n_{j,l})\log(t_2)+n_{j,l}\log|c_{m_l}|+...+n_{j,1}\log|c_{m_1}|}{n_{j,l}\log|c_{m_l}|+...+n_{j,y}\log|c_{m_y}|}\nonumber\\
&\qquad \qquad < 1+\frac{(l+n_{j,1}+\ldots+n_{j,l})\log(t_2)+n_{j,y-1}\log|c_{m_l}|+...+n_{j,1}\log|c_{m_1}|}{n_{j,l}\log|c_{m_l}|}\nonumber\\
&\qquad \qquad <1+\frac{(l+n_{j,1}+\ldots+n_{j,l})\log(t_2)+\widehat{q_{M-1}}y\log|c_{m_{y-1}}|}{n_{j,l}\log|c_{m_l}|}\nonumber\\
&\qquad \qquad <1+\frac{(l+n_{j,1}+\ldots+n_{j,l-1})\log(t_2)+\widehat{q_{M-1}}M\log|c_M|}{\log|c_{m_l}|}+\frac{1}{\log |c_{m_l}|}\nonumber\\
&\qquad \qquad <1+\frac{(l+n_{j,1}+\ldots+n_{j,l-1})\log(t_2)}{\log |c_{m_l}|}+\frac{\widehat{q_{M-1}}Mk_M(\log L+\delta_1)}{k_{m_l}(\log L-\delta_1)}+\frac{1}{\log |c_M|}\nonumber\\
&\qquad \qquad <1+\frac{2m_l\widehat{q_{M-1}}}{k_{m_l}\log L}+\frac{10}{\varphi^{M-1}(\varphi-1)\log L}+\frac{\widehat{q_{M-1}}Mk_M(\log L+\delta_1)}{k_{m_l}(\log L-\delta_1)}+\frac{1}{\log |c_M|}\nonumber\\
&\qquad \qquad <1+\frac{2m_l\widehat{q_{M-1}}}{k_{m_l}\log L}+\frac{10}{\varphi^{M-1}(\varphi-1)\log L}+\frac{3\widehat{q_{M-1}}Mk_M}{k_{m_l}}+\frac{2}{k_M\log L}\nonumber\\
&\qquad \qquad <1+\delta_2\label{epsilon2bound}.
\end{align}
Combining \eqref{gbounds}, \eqref{epsilon1bounds}, and \eqref{epsilon2bound}, we thus have
\begin{align*}
\frac{\log |g_{n_j}|}{n_j}&<\frac{(1+\delta_2)(n_{j,l}\log|c_{m_l}|+...+n_{j,y}\log|c_{m_y}|)}{n_{j,l}k_{m_l}+...+n_{j,y}k_{m_y}}\\
&<(\log L+\delta_1)(1+\delta_2)\\
&<\log L+\epsilon.
\end{align*}
Also, since $t_1\leq 1$, by \eqref{epsiloncond6}, \eqref{epsiloncond3}, and \eqref{nlog}, we have
\begin{align}
&\frac{(l+n_{j,1}+\ldots+n_{j,l})\log(t_1)+n_{j,l}\log|c_{m_l}|+...+n_{j,1}\log|c_{m_1}|}{n_{j,l}\log|c_{m_l}|+...+n_{j,y}\log|c_{m_y}|}\nonumber\\
&\qquad \qquad >1-\frac{(l+n_{j,1}+\ldots+n_{j,l})\log(t_1^{-1})}{n_{j,l}\log|c_{m_l}|}\nonumber\\
&\qquad \qquad >1-\frac{(l+n_{j,1}+\ldots+n_{j,l-1})\log(t_1^{-1})}{\log|c_{m_l}|}-\frac{2\log(t_1^{-1})}{k_M\log L}\nonumber\\
&\qquad \qquad >1-\frac{2m_l\widehat{q_{M-1}}\log(t_1^{-1})}{k_{m_l}\log L}-\frac{10\log(t_1^{-1})}{\varphi^{M-1}(\varphi-1)\log L}-\frac{2\log(t_1^{-1})}{k_M\log L}\nonumber\\
&\qquad \qquad >1-\delta_3.\label{epsilon3bound}
\end{align}
Also
\begin{align}
\frac{n_j}{n_{j,l}k_{m_l}+...+n_{j,y}k_{m_y}}&=1+\frac{n_{j,y-1}k_{m_{y-1}}+\ldots+n_{j,1}k_{m_1}}{n_{j,l}k_{m_l}+...+n_{j,y}k_{m_y}}\nonumber\\
&<1+\frac{\widehat{q_{M-1}}yk_{m_{y-1}}}{k_{m_l}}\nonumber\\
&<1+\frac{\widehat{q_{M-1}}Mk_M}{k_{m_l}}\nonumber\\
&<1+\delta_4\label{epsilon4bound}
\end{align}
by \eqref{epsiloncond4}. Combining \eqref{gbounds}, \eqref{epsilon1bounds}, \eqref{epsilon3bound}, and \eqref{epsilon4bound}, we have
\begin{align*}
\frac{\log |g_{n_j}|}{n_j}&>\frac{(1-\delta_3)(n_{j,l}\log|c_{m_l}|+...+n_{j,y}\log|c_{m_y}|)}{(1+\delta_4)(n_{j,l}k_{m_l}+...+n_{j,y}k_{m_y})}\\
&>\frac{(1-\delta_3)(\log L-\delta_1)}{(1+\delta_4)}\\
&>\log L-\epsilon.
\end{align*}
\end{proof}
\begin{lemma}\label{limitratios1}
Let $q_m$, $a_m$ and $c_m$ be as defined in Notation \ref{P_1P_2}. We have $\lim_{m\rightarrow\infty}\frac{a_m}{c_m}$ exists, is between $-1$ and $1$, and is irrational.
\end{lemma}
\begin{proof}
We have
\begin{equation*}
\lim_{m\rightarrow\infty}\frac{\log|c_m|}{k_m}=\log L>0.
\end{equation*}
Thus there exists $L'>1$ such that for all sufficiently large $m\in\mathbb{N}$, we have
\begin{equation*}
|c_{m-1}|>L'^{k_{m-1}}.
\end{equation*}
Let
\begin{equation*}
P_{m-1}^{q_{m-2}-1}P_{m-2}=:
\begin{bmatrix}
a_m' & b_m'\\
c_m' & d_m'
\end{bmatrix}.
\end{equation*}
Then for $m\in\mathbb{N}$ sufficiently large, we have
\begin{align*}
\left|\frac{a_m}{c_m}-\frac{a_{m-1}}{c_{m-1}}\right|&=\left|\frac{a_{m-1}a_{m-2}'+b_{m-1}c_{m-2}'}{c_{m-1}a_{m-2}'+d_{m-1}c_{m-2}'}-\frac{a_{m-1}}{c_{m-1}}\right|\\
&=\left|\frac{\left(b_{m-1}-\dfrac{a_{m-1}d_{m-1}}{c_{m-1}}\right)c_m'}{c_{m-1}a_{m-2}'+d_{m-1}c_{m-2}'}\right|
=\frac{|c_m'|}{|c_m||c_{m-1}|}\\
&\leq\frac{|c_m'|}{t_1|c_m'||c_{m-1}|^2}
\leq\frac{1}{t_1L'^{2k_{m-1}}}.
\end{align*}
By a geometric series argument, using \eqref{kngrowth}, the sequence $\frac{a_m}{c_m}$ is Cauchy and so converges. The fact that the limit is between $-1$ and $1$ follows from Lemma \ref{ratios1}. It remains to show the limit is irrational. Suppose for a contradiction that it is rational and let it be $\frac{a}{b}$ where $a,b\in\mathbb{N}$. Let $N\in\mathbb{N}$ be sufficiently large so that for all $m\geq N$, the above inequality holds and
$L'^{-2k_{m-1}}<\frac{1}{2}$.
Then for all $m\geq N$, we have
\begin{align*}
\left|\frac{a}{b}-\frac{a_m}{c_m}\right|&\leq\sum_{i=m}^{\infty}\left|\frac{a_{i+1}}{c_{i+1}}-\frac{a_i}{c_i}\right|
<\sum_{i=m}^{\infty}\frac{1}{t_1L'^{2k_i}}
=\frac{1}{t_1L'^{2k_m}}\sum_{i=0}^{\infty}L'^{2k_m-2k_{m+i}}\\
&=\frac{1}{t_1L'^{2k_m}}\sum_{i=0}^{\infty}\prod_{j=0}^{i-1}L'^{2k_{m+j}-2k_{m+j+1}}\\
&<\frac{1}{t_1L'^{2k_m}}\sum_{i=0}^{\infty}\prod_{j=0}^{i-1}L'^{-2k_{m+j-1}}\\
&<\frac{1}{t_1L'^{2k_m}}\sum_{i=0}^{\infty}\left(\frac{1}{2}\right)^i\\
&=\frac{2}{t_1L'^{2k_m}}.
\end{align*}
Thus
\begin{equation*}
\frac{|ac_m-ba_m|}{|bc_m|}<\frac{2}{t_1L'^{2k_m}}.
\end{equation*}
Suppose that $\frac{a}{b}\neq\frac{a_m}{c_m}$. Then we have
\begin{equation*}
\frac{1}{|bc_m|}<\frac{2}{t_1L'^{2k_m}}.
\end{equation*}
Thus
\begin{equation*}
\frac{t_1L'^{2k_m}}{2|b|}<|c_m|
\text{\ \ \ or\ \ \ }
\left(\frac{t_1}{2|b|}\right)^{1/k_m}L'^2<|c_m|^{1/k_m}.
\end{equation*}
Thus if there are infinitely many $m\in\mathbb{N}$ such that $\frac{a}{b}\neq\frac{a_m}{c_m}$, then we have $L'^2\leq L'$, a contradiction since $L>1$. So for sufficiently large $m\in\mathbb{N}$, we have $\frac{a}{b}=\frac{a_m}{c_m}$. But for all $m\in\mathbb{N}$, we have $\gcd(a_m,c_m)=1$ and $\lim_{m\rightarrow\infty}|a_m|=\lim_{m\rightarrow\infty}|c_m|=\infty$ and so this cannot be the case either. Thus the limit must be irrational.
\end{proof}
\begin{remark}
Let $q_m$, $a_m$ and $c_m$ as defined in Notation \ref{P_1P_2}. We will denote
\begin{equation*}
M:=\lim_{m\rightarrow\infty}\frac{a_m}{c_m}.
\end{equation*}
\end{remark}
\begin{lemma}\label{n_jratiolimit}
Let $q_m$, $P_m$, $Q_n$, $e_n$, $f_n$, $g_n$, and $h_n$ be as defined in Notation \ref{P_1P_2}. Let $n_1,n_2,\ldots,n_j,\ldots,$ be the list of natural numbers such that for each $n_j$, there exists $2\leq m_1<m_2<...<m_l$ and $n_{j,1},\ldots,n_{j,l}\in\mathbb{N}$ with the following properties:
\begin{enumerate}
\item for all $1\leq i\leq l$, we have $n_{j,i}\leq q_{m_i-1}$
\item for all $2\leq i\leq l$, if $n_{j,i}=q_{m_i-1}$, then $m_{i-1}+2\leq m_i$
\item $Q_n=(P_{m_l})^{n_{j,l}}(P_{m_{l-1}})^{n_{j,l-1}}...(P_{m_1})^{j,n_1}$.
\end{enumerate}

We have $\lim_{j\rightarrow\infty}\frac{e_{n_j}}{g_{n_j}}$ and $\lim_{j\rightarrow\infty}\frac{f_{n_j}}{h_{n_j}}$ both exist and are equal to $M$.
\end{lemma}
\begin{proof}
By Lemma \ref{limitratios1}, we have that
\begin{equation*}
\lim_{m\rightarrow\infty}\frac{e_{k_m}}{g_{k_m}}
\end{equation*}
exists and is equal to $M$. We will prove that the desired limit is $M$. Let $\epsilon>0$. Choose $N\in\mathbb{N}$ such that for all $m\geq N$, we have
\begin{equation*}
\left|\frac{e_{k_m}}{g_{k_m}}-M\right|<\frac{\epsilon}{2}
\text{\ \ \ and\ \ \ }
\frac{r_4^2}{(r_4^2-r_3^2)|g_{k_m}|}<\frac{\epsilon}{2}.
\end{equation*}
Let $n_j\geq k_N$. Then $k_{m+1}>n_j\geq k_m$ for some $m\geq N$. We have $Q_{n_j}=Q_{k_m}Q_{n_j-k_m}$. Thus we have the following using Lemma \ref{prodrelc_m}:
\begin{align*}
\left|\frac{e_{n_j}}{g_{n_j}}-\frac{e_{k_m}}{g_{k_m}}\right|&=\left|\frac{e_{k_m}e_{n_j-k_m}+f_{k_m}g_{n_j-k_m}}{g_{k_m}e_{n_j-k_m}+h_{k_m}g_{n_j-k_m}}-\frac{e_{k_m}}{g_{k_m}}\right|\\
&=\left|\frac{g_{n_j-k_m}(f_{k_m}g_{k_m}-e_{k_m}h_{k_m})}{g_{k_m}(g_{k_m}e_{n_j-k_m}+h_{k_m}g_{n_j-k_m})}\right|\\
&=\frac{|g_{n_j-k_m}|}{|g_{k_m}||g_{k_m}e_{n_j-k_m}+h_{k_m}g_{n_j-k_m}|}\\
&\leq\frac{|g_{n_j-k_m}|}{|g_{k_m}|(|h_{k_m}g_{n_j-k_m}|-|g_{k_m}e_{n_j-k_m}|)}\\
&<\frac{|g_{n_j-k_m}|}{|g_{k_m}|\left(\left|h_{k_m}g_{n_j-k_m}\right|-\dfrac{r_3^2}{r_4^2}\left|h_{k_m}g_{n_j-k_m}\right|\right)}\\
&=\frac{r_4^2}{(r_4^2-r_3^2)|g_{k_m}h_{k_m}|}\\
&\leq\frac{r_4^2}{(r_4^2-r_3^2)|g_{k_m}|}\\
&<\frac{\epsilon}{2}.
\end{align*}
Thus
\begin{equation*}
\left|\frac{e_{n_j}}{g_{n_j}}-M\right|\leq\left|\frac{e_{n_j}}{g_{n_j}}-\frac{e_{k_m}}{g_{k_m}}\right|+\left|\frac{e_{k_m}}{g_{k_m}}-M\right|<\epsilon.
\end{equation*}
Thus
\begin{equation*}
\lim_{j\rightarrow\infty}\frac{e_{n_j}}{g_{n_j}}=M.
\text{\ \ \ and\ \ \ }
\left|\frac{e_{n_j}}{g_{n_j}}-\frac{f_{n_j}}{h_{n_j}}\right|=\frac{1}{|g_{n_j}h_{n_j}|},
\end{equation*}
from which the rest follow using Lemma \ref{prodrelc_m}.
\end{proof}
\begin{remark}
For the rest of the paper, we will assume that $G_1\neq\frac{-G_2}{M}$.
\end{remark}
\begin{proposition}\label{partiallimit}
Let $q_m$, $P_m$ and $Q_n$ be defined as in Notation \ref{P_1P_2}. Let $n_1,n_2,\ldots,n_j,\ldots,$ be the list of natural numbers such that for each $n_j$, there exists $2\leq m_1<m_2<...<m_l$ and $n_{j,1},\ldots,n_{j,l}\in\mathbb{N}$ with the following properties:
\begin{enumerate}
\item for all $1\leq i\leq l$, we have $n_{j,i}\leq q_{m_i-1}$
\item for all $2\leq i\leq l$, if $n_{j,i}=q_{m_i-1}$, then $m_{i-1}+2\leq m_i$
\item $Q_{n_j}=(P_{m_l})^{n_{j,l}}(P_{m_{l-1}})^{n_{j,l-1}}...(P_{m_1})^{j,n_1}$.
\end{enumerate}

Let 
\begin{equation*}
[G_1,G_2]Q_n=[G_{n+1},G_{n+2}]
\end{equation*}
for all $n\in\mathbb{N}$. We have
\begin{equation*}
\lim_{j\rightarrow\infty}|G_{n_j+1}|^{1/(n_j+1)}=\lim_{j\rightarrow\infty}|G_{n_j+2}|^{1/(n_j+2)}=L.
\text{\ \ \ and\ \ \ }
\lim_{j\rightarrow\infty}\frac{G_{n_j+1}}{G_{n_j+2}}=M.
\end{equation*}
\end{proposition}
\begin{proof}
For all $n_j$, we have $G_1e_{n_j}+G_2g_{n_j}=G_{n_j+1}$. By Lemma \ref{prodrelc_m}, $g_{n_j}\neq 0$. Thus $\frac{G_1e_{n_j}}{g_{n_j}}+G_2=\frac{G_{n_j+1}}{g_{n_j}}$. By Lemma \ref{n_jratiolimit}, we have
\begin{equation*}
\lim_{j\rightarrow\infty}\frac{G_{n_j+1}}{g_{n_j}}=MG_1+G_2.
\end{equation*}
Since $G_1\neq\frac{-G_2}{M}$, we have the limit is nonzero and so
\begin{equation*}
\lim_{j\rightarrow\infty}|G_{n_j+1}|^{1/(n_j+1)}=L.
\end{equation*}
follows from Proposition \ref{rootgnj}. The limit
\begin{equation*}
\lim_{j\rightarrow\infty}|G_{n_j+2}|^{1/(n_j+2)}=L
\end{equation*}
follows similarly. Also, for all $j\in\mathbb{N}$, we have
\begin{align*}
\frac{G_{n_j+1}}{G_{n_j+2}}&=\frac{G_1e_{n_j}+G_2g_{n_j}}{G_1f_{n_j}+G_2h_{n_j}}\\
&=\frac{e_{n_j}}{f_{n_j}}\cdot\frac{G_1+\dfrac{G_2g_{n_j}}{e_{n_j}}}{G_1+\dfrac{G_2h_{n_j}}{f_{n_j}}}.
\end{align*}
Applying Lemma \ref{n_jratiolimit} gives us the third limit with the observation that
\begin{equation*}
\lim_{j\rightarrow\infty}G_1+\frac{G_2h_{n_j}}{f_{n_j}}=G_1+\frac{G_2}{M}\neq 0.
\end{equation*}
\end{proof}
\begin{proof}[Proof of Theorem \ref{bigthm} for Case $2$]
Let $n_1,n_2,\ldots,n_j,\ldots,$ be the list of natural numbers such that for each $n_j$, there exists $2\leq m_1<m_2<...<m_l$ and $n_{j,1},\ldots,n_{j,l}\in\mathbb{N}$ with the following properties:
\begin{enumerate}
\item for all $1\leq i\leq l$, we have $n_{j,i}\leq q_{m_i-1}$
\item for all $2\leq i\leq l$, if $n_{j,i}=q_{m_i-1}$, then $m_{i-1}+2\leq m_i$
\item $Q_n=(P_{m_l})^{n_{j,l}}(P_{m_{l-1}})^{n_{j,l-1}}...(P_{m_1})^{j,n_1}$.
\end{enumerate}
For all $j\in\mathbb{N}$, we have $n_{j+1}-n_j<k_2$ by Lemma \ref{Q_nproductP_n}. We can deduce that there exists $C>0$ such that for all $n\in\mathbb{N}$ there exists $n_j<n$ and integers $C_1,C_2<C$ such that $G_n=C_1G_{n_j+1}+C_2G_{n_j+2}$. By Proposition \ref{partiallimit}, we can deduce that
\begin{equation*}
\limsup_{n\rightarrow\infty}|G_n|^{1/n}=L.
\end{equation*}
Also, out of all of the finite possibilities for $C_1$ and $C_2$, we observe that $MC_1+C_2\neq 0$ since $M$ is irrational. Let $M'$ denote the minimal possible value of $C_1+MC_2$ in absolute value. Then $M'>0$. By Proposition \ref{partiallimit}, for all $1\leq t<k_2$, we have
\begin{equation*}
\liminf_{j\rightarrow\infty}\frac{|G_{n_j+t}|}{|G_{n_j+1}|}\geq M'.
\end{equation*}
It follows that
\begin{equation*}
\liminf_{n\rightarrow\infty}|G_n|^{1/n}=L.
\end{equation*}
\end{proof}
\section{Future Work}
There are a couple of different directions this research can go in. The first direction involves studying the growth rates of generalised Fibonacci sequences produced by other patterns of words. Here we have examined Sturmian words, but there are other types of words as well. Some examples are words that follow a Thue-Morse pattern, as well as other morphisms. We could even try removing the condition $j,k\geq 2$ in Example \ref{entriessize}.

The other direction involves trying to calculate the exact growth rate of certain random Fibonacci sequences produced from words following such patterns and seeing how close to Viswanath's constant we can get. McLellan \cite{mclellan} used words following a periodic pattern to create a new way of calculating Viswanath's constant. By adding in new patterns into her method, we may be able to calculate Viswanath's constant even more accurately. We even might be able to calculate its exact value or at least shed some light on its nature (for example, if it's irrational, transcendental, etc.).

\section{Acknowledgements}
The research of Kevin G. Hare is supported in part by NSERC grant 2019-03930.
The research of J.C. Saunders is supported by an Azrieli International Postdoctoral Fellowship, as well as a Postdoctoral Fellowship at the University of Calgary.

\section{Competing Interest Statement}
The authors have no competing interests to declare.

\end{document}